\documentclass[oneside]{amsart} 
\usepackage{a4wide}
\usepackage{amsmath,amsthm,amscd}
\usepackage{colonequals}
\usepackage{enumerate}
\usepackage{amssymb,amsfonts} %
\input xy
\xyoption{all}
\numberwithin{equation}{section}
\theoremstyle{definition}
\newtheorem{dfn}{Definition}[section]
\newtheorem{example}[dfn]{Example}
\newtheorem{rem}[dfn]{Remark}
\theoremstyle{plain}
\newtheorem{thm}[dfn]{Theorem}

\newtheorem{cor}[dfn]{Corollary}
\newtheorem{lem}[dfn]{Lemma}

\newtheorem*{thmA}{Theorem A}       
\newcommand\al{\alpha} 
\newcommand\ep{\varepsilon}
 \newcommand\la{\lambda}
  \newcommand\si{\sigma}

\newcommand\C{\mathbb{C}}\newcommand\Z{\mathbb{Z}}
\newcommand\K{\Bbbk}

\newcommand{\TGWC}[4]{\mathcal{C}_{#4}({#1},{#2},{#3})}  
\newcommand{\TGWA}[4]{\mathcal{A}_{#4}({#1},{#2},{#3})}
\newcommand{\TGWI}[4]{\mathcal{I}_{#4}({#1},{#2},{#3})}

\newcommand{\Lc}[1]{L_{#1}}
\newcommand{\Nc}[1]{N_{#1}}
\newcommand{\PM}[2]{\mathrm{PM}_{#1}({#2})}
\newcommand{\JJ}[1]{\mathcal{J}_{#1}}

\renewcommand{\hat}{\widehat}
\renewcommand{\tilde}{\widetilde}

\newcommand{\TGW}[2]{\mathsf{TGW}_{#1}({#2})}

\newcommand{\AlgG}[2]{\text{${#2}$-$\mathsf{GrAlg}_{#1}$}}

\newcommand{\Af}[1]{\mathcal{A}_{#1}}
\newcommand{\Ff}{\mathcal{F}}

\DeclareMathOperator{\Aut}{Aut}

\date{}
\begin{document}
\title{On the Consistency of Twisted Generalized Weyl Algebras}
\author{Vyacheslav Futorny}
\address{ Instituto de Matem\'atica e Estat\'\i stica,
Universidade de S\~ao Paulo,  S\~ao Paulo, 
 Brasil}
 \email{futorny@ime.usp.br}
 \author{Jonas T. Hartwig}
\address{Department of Mathematics, Stanford University, Stanford, CA, USA}
\email{jonas.hartwig@gmail.com }
\maketitle
\begin{abstract}
A twisted generalized Weyl algebra $A$ of degree $n$ depends on a base algebra $R$, $n$ commuting
automorphisms $\si_i$ of $R$, $n$ central elements $t_i$ of $R$
and on some additional scalar parameters.

In \cite[Lemma 1]{MT99} it is claimed that certain consistency conditions for $\si_i$ and $t_i$
are sufficient for the algebra to be nontrivial. However, in this paper we give an example which shows that
this is false. We also correct the statement by finding a new set of consistency conditions
and prove that the old and new conditions together are necessary and sufficient for the base
algebra $R$ to map injectively into $A$.
In particular they are sufficient for the algebra $A$ to be nontrivial.

We speculate that these consistency relations may play a role in other areas of mathematics, analogous to the role played by the Yang-Baxter equation in the theory of integrable systems.

\end{abstract}

\tableofcontents

\section{Introduction}


Let $R$ be an algebra over a commutative ring $\K$, $\si_1,\ldots,\si_n$  commuting $\K$-algebra automorphisms of $R$,
 $t_1,\ldots,t_n$  elements from the center of $R$, and
$\mu_{ij}$ an $n\times n$ matrix of invertible scalars from $\K$. To this data one associates
a \emph{twisted generalized Weyl algebra} $\TGWA{R}{\si}{t}{\mu}$, an associative
$\Z^n$-graded algebra (see Section \ref{sec:tgwadef} for definition). These algebras were
introduced by Mazorchuk and Turowska \cite{MT99} and they are generalizations of the
much studied generalized Weyl algebras, defined independently by Bavula \cite{B},
 Jordan \cite{Jordan1993}, and Rosenberg \cite{Rb} (there called \emph{hyperbolic rings}).

Simple weight modules over twisted generalized Weyl algebras have been studied in \cite{MT99},\cite{MPT},\cite{Ha06}.
In \cite{MT02} the authors classified bounded and unbounded $\ast$-representations over
twisted generalized Weyl algebras.
Interesting examples of twisted generalized Weyl algebras were given in \cite{MPT}.
In \cite{Ha09} new examples of twisted generalized Weyl algebras were constructed
from symmetric Cartan matrices.

It was claimed in \cite[Lemma 1]{MT99} (for the case when $\mu_{ij}=1\,\forall i,j$) and implicitly in \cite[Eq. (1)]{MT02} (for arbitrary $\mu_{ij}$) that a TGWA $\TGWA{R}{\si}{t}{\mu}$ is a nontrivial ring if the following relations are satisfied:
\begin{equation}\label{eq:tgwa_consistency1}
t_it_j=\mu_{ij}\mu_{ji}\si_i^{-1}(t_j)\si_j^{-1}(t_i),\qquad i\neq j.
\end{equation}
However in this paper we give an example (Example \ref{ex:trivialtgwa_1}) of a TGWA, $A=\TGWA{R}{\si}{t}{\mu}$, such that even though the datum $(R,\si,t,\mu)$ satisfies \eqref{eq:tgwa_consistency1}, the algebra $A$ is the trivial ring $\{0\}$. This shows that, in fact, \eqref{eq:tgwa_consistency1} is not sufficient for a TGWA to be a nontrivial ring.

Our main result in this paper is the discovery of a new consistency relation,
which together with \eqref{eq:tgwa_consistency1} gives a sufficient condition
for a TGWA to be a nontrivial ring. The precise statement is the following.

\begin{thmA}
Let $\K$ be a commutative unital ring,
$R$ be an associative $\K$-algebra, $n$ a positive integer,
$t=(t_1,\ldots,t_n)$ be an $n$-tuple of regular central elements of $R$,
$\si:\Z^n\to\Aut_\K(R)$ a group homomorphism, $\mu_{ij}$ ($i,j=1,\ldots,n, i\neq j$) invertible
elements from $\K$, and
$\TGWA{R}{\si}{t}{\mu}$ the corresponding twisted generalized Weyl algebra,
equipped with the canonical homomorphism of $R$-rings
$\rho:R\to \TGWA{R}{\si}{t}{\mu}$.
Then the following two statements are equivalent:
\begin{enumerate}[{\rm (a)}]
\item $\rho$ is injective,
\item the following two sets of relations are satisfied in $R$:
\begin{alignat}{2}
\label{eq:TGWA_consistency_relations1}
\si_i\si_j(t_it_j) &= \mu_{ij}\mu_{ji}\si_i(t_i)\si_j(t_j),
  &\qquad \forall i,j=1,\ldots,n,\, i\neq j, \\
\label{eq:TGWA_consistency_relations2}
t_j\si_i\si_k(t_j) &= \si_i(t_j)\si_k(t_j),
  &\qquad \forall i,j,k=1,\ldots,n,\, i\neq j\neq k\neq i.
\end{alignat} 
 \end{enumerate}
In particular, if \eqref{eq:TGWA_consistency_relations1} and \eqref{eq:TGWA_consistency_relations2}
are satisfied, then $\TGWA{R}{\si}{t}{\mu}$ is nontrivial iff $R$ is nontrivial.
Moreover, neither of the two conditions \eqref{eq:TGWA_consistency_relations1} and
 \eqref{eq:TGWA_consistency_relations2} imply the other.
\end{thmA}

One may note that \eqref{eq:TGWA_consistency_relations1} and \eqref{eq:TGWA_consistency_relations2} may be expressed as
the following identities in the localization $T^{-1}R$,  where $T$ is the multiplicative submonoid of $R$ generated by all $\sigma_g(t_i)$ for all $g\in\Z^n, i=1,\ldots,n$:
\begin{alignat}{2}
\tau_{ij}\tau_{ji}&=1,
  &\qquad \forall i,j=1,\ldots,n,\, i\neq j, \\
\si_k(\tau_{ij})&=\tau_{ij},
  &\qquad \forall i,j,k=1,\ldots,n,\, i\neq j\neq k\neq i,
\end{alignat}
where
\begin{equation}
\tau_{ij}=\frac{\mu_{ji}\si_j(t_j)}{\si_i\si_j(t_j)}.
\end{equation}

%
%
%
%
%
%

We note that the consistency relations \eqref{eq:TGWA_consistency_relations1},\eqref{eq:TGWA_consistency_relations2} play an analogous role for twisted generalized Weyl algebras as the quantum Yang-Baxter equation plays for Zamolodchikov algebras in factorized $S$-matrix models \cite[Section~1.1.1]{GomRuiSie1996}, and for Faddeev-Reshetikhin-Takhtajan algebras.
Therefore we pose the following questions.
\begin{enumerate}[{\rm (a)}]
\item Is there a direct relation between the consistency relations
\eqref{eq:TGWA_consistency_relations1},\eqref{eq:TGWA_consistency_relations2}
 and the quantum Yang-Baxter equation?
\item Can one construct some physical model for which
relations \eqref{eq:TGWA_consistency_relations1},\eqref{eq:TGWA_consistency_relations2} appear as a condition for the model to be integrable (exactly solvable)?
\end{enumerate}
We leave these open questions to be addressed in future publications.

The structure of the paper is as follows. We first give some constructions of twisted generalized Weyl algebras in full generality. After recalling the definition in Section \ref{sec:tgwadef}, we show in Section \ref{sec:functoriality} that the construction of a twisted generalized Weyl
algebra is functorial in the initial data. In Sections \ref{sec:tgwaquotients} and Section \ref{sec:tgwalocalization}
we study two natural operations: taking quotients and localizing.
In Section \ref{sec:muconsistency} we prove Theorem~A.
In Section \ref{sec:weaklymuconsistent} (Theorem \ref{thm-main-1}(a)), we show that consistent TGW algebras with all $t_i$ invertible are in fact $\Z^n$-crossed product algebras over $R$. Thus these TGW algebras may be thought of as producing solutions to equations \eqref{eq:crossedproducteqs} required for the $\Z^n$-action and twisted $2$-cocycle map to produce nontrivial crossed product algebras.
We also define a notion of weak consistency and prove in Theorem \ref{thm-main-1}(b) that there is an embedding of any regular, weakly consistent TGW algebra into an associated $\Z^n$-crossed product algebra.

\subsection*{Acknowledgements}
This work was carried out during the second author's postdoc at IME-USP,
funded by FAPESP, processo 2008/10688-1. 
The first  author is supported in part by the CNPq grant (301743/
2007-0) and by the Fapesp grant (2005/60337-2). 


\subsection*{Notation and conventions}
By ``ring'' (``algebra'') we mean unital associative ring (algebra).
All ring and algebra morphisms are required to be unital.
By ``ideal'' we mean two-sided ideal unless otherwise stated.
An element $x$ of a ring $R$ is said to be \emph{regular in $R$} if
for all nonzero $y\in R$ we have $xy\neq 0$ and $yx\neq 0$.
The set of invertible elements in a ring $R$ will be denoted by $R^\times$.

Let $R$ be a ring.
Recall that an $R$-ring is a ring $A$ together with a ring morphism $R\to A$.
Let $X$ be a set. Let $RXR$ be the free $R$-bimodule on $X$.
The free $R$-ring $F_R(X)$ on $X$ is defined as the tensor algebra of the free
$R$-bimodule on $X$: $F_R(X)=\oplus_{n\ge 0} (RXR)^{\otimes_R n}$
where $(RXR)^{\otimes_R 0}=R$ by convention and the ring morphism $R\to F_R(X)$
is the inclusion into the degree zero component.

Throughout this paper we fix a commutative ring $\K$.
\section{Definition of TGW algebras} \label{sec:tgwadef}
We recall the definition of twisted generalized Weyl algebras \cite{MT99,MPT}.
Here we emphasize the initial data more than usual, which will be
useful in the next section to express the functoriality of the construction.

\begin{dfn}[TGW datum]
Let $n$ a positive integer.
A \emph{twisted generalized Weyl datum (over $\K$ of degree $n$)}
 is a triple $(R,\si,t)$ where
\begin{itemize}
\item $R$ is a unital associative $\K$-algebra,
\item $\si$ is a group homomorphism $\si:\Z^n\to\Aut_\K(R)$, $g\mapsto \si_g$,
\item $t$ is a function $t:\{1,\ldots,n\}\to Z(R)$, $i\mapsto t_i$.
\end{itemize} 
A \emph{morphism} between TGW data over $\K$ of degree $n$, 
\[\varphi:(R,\si,t)\to (R',\si',t')\]
is a $\K$-algebra morphism $\varphi: R\to R'$ such that
$\varphi \si_i=\si_i'\varphi$ and $\varphi(t_i)=t_i'$ for all $i\in\{1,\ldots,n\}$.
We let $\TGW{n}{\K}$ denote the category whose objects are the TGW data
over $\K$ of degree $n$ and morphisms are as above.
\end{dfn}
For $i\in\{1,\ldots,n\}$ we put $\si_i=\si_{e_i}$, where $\{e_i\}_{i=1}^n$ is the standard
$\Z$-basis for $\Z^n$.
A \emph{parameter matrix (over $\K^\times$ of size $n$)} is an $n\times n$
matrix $\mu=(\mu_{ij})_{i\neq j}$ without diagonal where $\mu_{ij}\in\K^\times\,\forall i\neq j$. The set of all parameter matrices over $\K^\times$ of size $n$
will be denoted by $\PM{n}{\K}$.


\begin{dfn}[TGW construction]
Let $n\in\Z_{>0}$, $(R,\si,t)$ be an object in $\TGW{n}{\K}$, and $\mu\in\PM{n}{\K}$.
The \emph{twisted generalized Weyl construction with parameter matrix $\mu$
associated to the TGW datum $(R,\si,t)$} is denoted by $\TGWC{R}{\si}{t}{\mu}$
and is defined as the free $R$-ring on the set $\{x_i,y_i\mid i=1,\ldots,n\}$
modulo the two-sided ideal generated by the following set of elements:
\begin{subequations}\label{eq:tgwarels}
\begin{alignat}{3}
\label{eq:tgwarels1}
x_ir  &-\si_i(r)x_i,  &\quad y_ir&-\si_i^{-1}(r)y_i, 
 &\quad \text{$\forall r\in R,\, i\in\{1,\ldots,n\}$,} \\
\label{eq:tgwarels2}
y_ix_i&-t_i, &\quad x_iy_i&-\si_i(t_i),
 &\quad \text{$\forall i\in\{1,\ldots,n\}$,} \\
\label{eq:tgwarels3}
&&\quad x_iy_j&-\mu_{ij}y_jx_i,
 &\quad \text{$\forall i,j\in\{1,\ldots,n\},\, i\neq j$.}
\end{alignat}
\end{subequations}
\end{dfn}
The images in $\TGWC{R}{\si}{t}{\mu}$ of the elements $x_i, y_i$  will
be denoted by $\hat X_i, \hat Y_i$ respectively.
The ring $\TGWC{R}{\si}{t}{\mu}$ has a $\Z^n$-gradation
given by requiring $\deg \hat X_i=e_i, \deg \hat Y_i=-e_i, \deg r=0\, \forall r\in R$.
Let $\TGWI{R}{\si}{t}{\mu}\subseteq \TGWC{R}{\si}{t}{\mu}$ be 
the sum of all graded ideals $J\subseteq \TGWC{R}{\si}{t}{\mu}$ having
zero intersection with the degree zero component, i.e. such that
$\TGWC{R}{\si}{t}{\mu}_0\cap J=\{0\}$.
It is easy to see that $\TGWI{R}{\si}{t}{\mu}$ is the unique maximal graded ideal having
zero intersection with the degree zero component.
\begin{dfn}[TGW algebra]
The \emph{twisted generalized Weyl algebra with parameter matrix $\mu$
associated to the TGW datum $(R,\si,t)$} is
denoted $\TGWA{R}{\si}{t}{\mu}$ and is defined as the
quotient $\TGWA{R}{\si}{t}{\mu}:=\TGWC{R}{\si}{t}{\mu} / \TGWI{R}{\si}{t}{\mu}$.
\end{dfn}
Since $\TGWI{R}{\si}{t}{\mu}$ is graded,
$\TGWA{R}{\si}{t}{\mu}$ inherits a $\Z^n$-gradation from $\TGWC{R}{\si}{t}{\mu}$.
The images in $\TGWA{R}{\si}{t}{\mu}$ of the elements $\hat X_i, \hat Y_i$ will be
denoted by $X_i, Y_i$.
By a \emph{monic monomial} in a TGW construction $\TGWC{R}{\si}{t}{\mu}$
(respectively TGW algebra $\TGWA{R}{\si}{t}{\mu}$)
we will mean a product of elements from
$\{\hat X_i, \hat Y_i\mid i=1,\ldots,n\}$
(respectively $\{X_i, Y_i\mid i=1,\ldots,n\}$).
 
The following statements are easy to check.
\begin{lem}\label{lem:easy}
\begin{enumerate}[{\rm (a)}]
\item \label{it:monomialgeneration}
$\TGWA{R}{\si}{t}{\mu}$ (respectively $\TGWC{R}{\si}{t}{\mu}$) is generated as a left and as a right $R$-module
by the monic monomials in $X_i, Y_i\,(i=1,\ldots,n)$ (respectively
$\hat X_i, \hat Y_i\, (i=1,\ldots,n)$).

\item
\label{it:A0}
The degree zero component of $\TGWA{R}{\si}{t}{\mu}$ is equal to
the image of $R$ under the natural map $\rho:R\to\TGWA{R}{\si}{t}{\mu}$.

\item\label{it:zerointersection}
Any nonzero graded ideal of $\TGWA{R}{\si}{t}{\mu}$ has nonzero intersection
with the degree zero component.
\end{enumerate}
\end{lem}

\begin{dfn}[$\mu$-Consistency]
Let $(R,\si,t)$ be a TGW datum over $\K$ of degree $n$ and
$\mu$ be a parameter matrix over $\K^\times$ of size $n$.
We say that $(R,\si,t)$ is \emph{$\mu$-consistent} if the
canonical map $\rho:R\to \TGWA{R}{\si}{t}{\mu}$ is injective.
\end{dfn}
Since $\TGWI{R}{\si}{t}{\mu}$ has zero intersection with the zero-component,
$(R,\si,t)$ is 
$\mu$-consistent iff the canonical map $R\to \TGWC{R}{\si}{t}{\mu}$ is 
injective.
Even in the cases when $\rho$ is not injective, we will often
view $\TGWA{R}{\si}{t}{\mu}$ as a left $R$-module and write for example
$rX_i$ instead of $\rho(r)X_i$.

\begin{dfn}[Regularity]
A TGW datum $(R,\si,t)$ is called \emph{regular} if $t_i$
is regular in $R$ for all $i$.
\end{dfn}

\begin{lem} \label{lem:tgwatgwc}
If $t_i\in R^\times$ for all $i$, then
the canonical projection $\TGWC{R}{\si}{t}{\mu}\to\TGWA{R}{\si}{t}{\mu}$
is an isomorphism.
\end{lem}
\begin{proof}
The algebra $\TGWC{R}{\si}{t}{\mu}$
is a $\Z^n$-crossed product algebra over its degree zero subalgebra,
since each homogeneous component 
contains an invertible element. Indeed
since $t_i\in R^\times$, each $X_i$ is invertible and thus
$X_1^{g_1}\cdots X_n^{g_n}$ has
degree $g$ and is invertible. Therefore any nonzero graded ideal in
$\TGWC{R}{\si}{t}{\mu}$ has nonzero intersection with the degree zero component, a
property which holds for any strongly graded ring, in particular
for crossed product algebras. Thus $\TGWI{R}{\si}{t}{\mu}=0$, which
proves the claim.
\end{proof}

We give an example of a TGW algebra which is the trivial ring. This shows the need for finding sufficient conditions on the TGW datum which will ensure the algebra is nontrivial. Finding such conditions is the goal of this paper.
\begin{example}\label{ex:trivialtgwa_1}
In \cite[Equation (2)]{MT02} it was observed that the relation
\begin{equation}\label{eq:ex:trivial}
\si_i\si_j(t_j)X_jX_i=\mu_{ij}\si_j(t_j)X_iX_j
\end{equation}
is satisfied in any TGW algebra $A=\TGWA{R}{\si}{t}{\mu}$ for every $i\neq j$.
There are two ways to commute a multiple of $X_kX_jX_i$ to a multiple of $X_iX_jX_k$
(starting by commuting $X_kX_j$ or $X_jX_i$ respectively),
using relation \eqref{eq:ex:trivial}. Namely,
\begin{align}
\label{eq:step2del1ex}
\si_j\si_k(t_k) \si_i\si_j\si_k(t_k) \si_i\si_j(t_j) X_kX_jX_i &=
\mu_{jk}\si_k(t_k) \mu_{ik}\si_j\si_k(t_k) \mu_{ij}\si_j(t_j) X_iX_jX_k
\intertext{and}
\label{eq:step2del2ex}
\si_i\si_j\si_k(t_j) \si_i\si_k(t_k) \si_i\si_j\si_k(t_k) X_kX_jX_i &=
\mu_{ij}\si_k\si_j(t_j) \mu_{ik}\si_k(t_k) \mu_{jk}\si_i\si_k(t_k) X_iX_jX_k.
\end{align}
 Combining \eqref{eq:step2del1ex} and \eqref{eq:step2del2ex} we get
\begin{multline}\label{eq:step2del3ex}
\big(\si_j\si_k(t_k) \si_i\si_j\si_k(t_k) \si_i\si_j(t_j) 
\mu_{ij}\si_k\si_j(t_j) \mu_{ik}\si_k(t_k) \mu_{jk}\si_i\si_k(t_k)
\\
-\mu_{jk}\si_k(t_k) \mu_{ik}\si_j\si_k(t_k) \mu_{ij}\si_j(t_j) 
\si_i\si_j\si_k(t_j) \si_i\si_k(t_k) \si_i\si_j\si_k(t_k) 
\big)X_iX_jX_k=0.
\end{multline}
Assume that $t_i$ is invertible in $R$ for every $i$. Thus, since $t_i=Y_iX_i$ in $A$, $X_i$ is invertible in $A$ for every $i$. Then \eqref{eq:step2del3ex} implies that, in the algebra $A$,
\begin{equation}\label{eq:asdf}
 \si_i\si_j(t_j) \si_k\si_j(t_j) - \si_j(t_j) \si_i\si_j\si_k(t_j) =0.
\end{equation}
(Note that we cannot conclude that \eqref{eq:asdf} must also hold in $R$, unless we know that the canonical map $\rho:R\to A$ is injective.)
The following is an example of a TGW algebra where 
\eqref{eq:tgwa_consistency1} holds but the left hand side of \eqref{eq:asdf} in fact is an invertible element of $A$, which forces $A$ to be the trivial ring.
Namely, let
\begin{itemize}
\item $n=3$, $\K=\C$, $\mu_{ij}=1$ for all $i\neq j$;
\item $R=\C[\al_{12}^{\pm 1},\al_{23}^{\pm 1},\al_{13}^{\pm 1},t_1^{\pm 1},t_2^{\pm 1},t_3^{\pm 1}]$, a Laurent polynomial algebra in $6$ variables;
\item $\si_1,\si_2,\si_3\in\Aut_\C(R)$, given by
\begin{equation}
\si_i(\al_{jk}) \colonequals 
  \begin{cases}
     \al_{jk}&\text{if $\{i,j,k\}\neq \{1,2,3\}$,}\\
    -\al_{jk}&\text{if $\{i,j,k\}=\{1,2,3\}$,}
  \end{cases}
\end{equation}
 for all $i,j,k\in\{1,2,3\}$ with $j<k$, and
\begin{equation}
  \si_i(t_j) \colonequals \al_{ij} t_j,
%
\end{equation}
for all $i,j\in\{1,2,3\}$, where
$\al_{ji}\colonequals\al_{ij}^{-1}$ for $i<j$ and $\al_{ii}\colonequals 1$ for $i=1,2,3$.
\end{itemize}
Then the automorphisms $\si_1,\si_2,\si_3$ commute with each other. Indeed, it is immediate that $\si_i\si_j(\al_{kl})=\si_j\si_i(\al_{kl})$ holds for all $i,j,k,l\in\{1,2,3\}$.
On the $t_i$ we have
\[\si_i\si_j(t_k)=\si_i(\al_{jk}t_k) =
\begin{cases}
-\al_{jk}\al_{ik} t_k , &\text{if $\{i,j,k\}=\{1,2,3\}$},\\
\al_{jk}\al_{ik} t_k, & \text{otherwise}.
\end{cases}
\]
Since the right hand side is symmetric in $i,j$ it follows that $\si_i$ and $\si_j$ commute on any $t_k$. 
Relation \eqref{eq:tgwa_consistency1} is easily checked.
However, for any $i,j,k$ such that $\{i,j,k\}=\{1,2,3\}$, the result of $\si_j^{-1}$ applied to the left hand side of \eqref{eq:asdf} equals
\begin{equation}\label{eq:trivialtgwa_invertible}
\si_i(t_j)\si_k(t_j)   -  \si_i\si_k(t_j)t_j
=\al_{ij}\al_{kj}t_j^2 - \si_i(\al_{kj}t_j)t_j 
= 2\al_{ij}\al_{kj}t_j^2
\end{equation}
which is invertible in $R$, hence the image is invertible in $A$. Thus \eqref{eq:asdf} implies that $A$ is the trivial ring $\{0\}$ despite the fact that \eqref{eq:tgwa_consistency1} holds.

Put in different terms, the situation for the TGW datum above is that the ideal in the free $R$-ring $F=F_R\big(\{x_i,y_i\mid i=1,\ldots,n\}\big)$ generated by the elements \eqref{eq:tgwarels} contains an invertible element of the algebra $R$ and consequently the quotient of $F$ by that ideal becomes trivial.

%

\end{example}

%
%


\section{Functoriality of the construction} \label{sec:functoriality}
We prove in this section that the construction of twisted generalized Weyl algebras is functorial in the initial data in a certain sense. This will be applied in the next two sections.


For a group $G$ we let
$\AlgG{\K}{G}$ denote the category of $G$-graded $\K$-algebras
where morphisms $\varphi:A\to B$ are graded $\K$-algebra homomorphisms:
$\varphi(A_g)\subseteq B_g$ for all $g\in G$.
We have the forgetful functor $\Ff:\TGW{n}{\K}\to\AlgG{\K}{\Z^n}$,
defined by $\Ff(R,\si,t)=R$ with trivial grading $R_0=R$.

\begin{thm}\label{thm:functoriality}
Let $n\in\Z_{>0}$ and $\mu$ be a parameter matrix over $\K^\times$ of size $n$.
\begin{enumerate}[{\rm (a)}]
\item \label{it:functoriality_a}
The construction of a twisted generalized Weyl algebra $\TGWA{R}{\si}{t}{\mu}$ from a TGW datum $(R,\si,t)$ defines a functor $\Af{\mu}:\TGW{n}{\K}\to \AlgG{\K}{\Z^n}$. That is, for any morphism $\varphi:(R,\si,t)\to (R',\si',t')$ in $\TGW{n}{\K}$ there is a morphism $\Af{\mu}(\varphi):\TGWA{R}{\si}{t}{\mu}\to\TGWA{R'}{\si'}{t'}{\mu}$ in $\AlgG{\K}{\Z^n}$ such that $\Af{\mu}$ preserves compositions and identity morphisms.
\item \label{it:functoriality_b}
 The family $\rho_\mu=\{\rho_{\mu,(R,\si,t)}\}$, where $(R,\si,t)$ ranges over the objects in $\TGW{n}{\K}$,
of canonical maps
 $\rho_{\mu,(R,\si,t)}:R\to\TGWA{R}{\si}{t}{\mu}$ defines a 
natural transformation
from the forgetful functor $\Ff$ to the functor $\Af{\mu}$ from part
(\ref{it:functoriality_a}).
That is,
given any morphism $\varphi:(R,\si,t)\to (R',\si',t')$ in $\TGW{n}{\K}$,
we have the following commutative diagram in $\AlgG{\K}{\Z^n}$:
\begin{equation}\label{eq:naturality}
\xymatrix@C=1.5cm{ 
 R \ar[r]^{\varphi} \ar[d]^{\rho_{\mu,(R,\si,t)}}
  & R'\ar[d]^{\rho_{\mu,(R',\si',t')}} \\
 \TGWA{R}{\si}{t}{\mu} \ar[r]^{\Af{\mu}(\varphi)} & \TGWA{R'}{\si'}{t'}{\mu}
}
\end{equation}
\item \label{it:functoriality_c}
For any morphism $\varphi:(R,\si,t)\to (R',\si',t')$ in $\TGW{n}{\K}$, the
algebra
$\TGWA{R'}{\si'}{t'}{\mu}$ is generated as a left and as a right $R'$-module
by the image of $\Af{\mu}(\varphi)$.
\end{enumerate}
\end{thm}
\begin{proof}
Let $\varphi:(R,\si,t)\to (R',\si',t')$
be a morphism of TGW data. 
Pulling back the canonical map $R'\to\TGWC{R'}{\si'}{t'}{\mu}$
through $\varphi$
we get a map $R\to\TGWC{R'}{\si'}{t'}{\mu}$. The universal property
of free $R$-rings implies that this map extends uniquely to a
morphism of $R$-rings from the
free $R$-ring on the set $\{x_i,y_i\mid i=1,\ldots,n\}$ by requiring
that $x_i\mapsto \hat X_i'$, $y_i\mapsto \hat Y_i'$ for all $i$, where
 $\hat X_i', \hat Y_i'$ denote the generators in $\TGWC{R'}{\si'}{t'}{\mu}$.
Since $\varphi$ is a morphism of TGW data, one verifies that
the ideal with generators \eqref{eq:tgwarels} lies in the kernel, thus
inducing a map $\hat\varphi$ such that the following diagram commutes:
\[
\xymatrix@C=1.5cm{ 
 R \ar[r]^{\varphi} \ar[d]
  & R'\ar[d] \\
 \TGWC{R}{\si}{t}{\mu} \ar[r]^{\hat\varphi} & \TGWC{R'}{\si'}{t'}{\mu}
}
\]
Let $\mathcal{I}=\TGWI{R}{\si}{t}{\mu}$ (respectively
$\mathcal{I}'=\TGWI{R'}{\si'}{t'}{\mu}$) denote the 
sum of all graded ideals in
$\TGWC{R}{\si}{t}{\mu}$ (respectively $\TGWC{R'}{\si'}{t'}{\mu}$)
having zero intersection with the degree zero component.
We need to verify that $\hat\varphi(\mathcal{I})\subseteq\mathcal{I}'$.
Let $a\in \mathcal{I}$. Since $\mathcal{I}$ is graded we can assume that $a$
is homogeneous. Let $d\in\Z^n$ be its degree. To prove that $\hat\varphi(a)\in\mathcal{I}'$ we show that the two-sided ideal generated by $\hat\varphi(a)$
has zero intersection with the zero component. Then, since $\mathcal{I}'$ is the sum of all such ideals, it must contain $\hat\varphi(a)$.
 Let $m>0$ and $b_j',c_j'\; (j=1,\ldots,m)$
be nonzero homogeneous elements of $\TGWC{R'}{\si'}{t'}{\mu}$
such that $d+\deg b_j'+\deg c_j' = 0$ for all $j$.
We must show that $\sum_{j} b_j'\varphi(a)c_j' = 0$.
Without loss of generality, the elements $b_j$ and $c_j$ are monomials.
Since $\hat\varphi(\hat X_i)=\hat X_i'$ and $\hat\varphi(\hat Y_i)=\hat Y_i'$
 it follows that
we have $b_j'=r_j'\hat\varphi(b_j)$, $c_j'=s_j'\hat\varphi(c_j)$ where $r_j',s_j'\in R'$ and $b_j,c_j\in\TGWC{R}{\si}{t}{\mu}$. Thus
\[\sum_j b_j'\hat\varphi(a)c_j' =
 \sum_j r_j'\hat\varphi(b_j)\hat\varphi(a)s_j'\hat\varphi(c_j) =
 \sum_j r_j'\tau_j(s_j') \varphi(b_jac_j) 
\]
where $\tau_j=\si_{d+\deg b_j}$. But $a\in\mathcal{I}$ so $b_jac_j\in\mathcal{I}$ and hence,
since $\deg(b_jac_j)=0$, it follows that $b_jac_j=0$ for each $j$. This proves that $\hat\varphi(a)$ generates a two-sided ideal having zero intersection with the degree zero component, and thus $\hat\varphi(a)\in\mathcal{I}'$.
Hence $\hat\varphi(\mathcal{I})\subseteq\mathcal{I'}$ and therefore $\hat\varphi$ induces a $\K$-algebra morphism $\Af{\mu}(\varphi):\TGWA{R}{\si}{t}{\mu}\to \TGWA{R'}{\si'}{t'}{\mu}$ such that the diagram \eqref{eq:naturality} commutes.
That $\Af{\mu}$ defines a functor is easy to check. This proves
part (a). Claim (b) follows from the construction of $\Af{\mu}(\varphi)$.

(c) By construction, the image of $\Af{\mu}(\varphi)$ contains all the elements
$X_i',Y_i'\;(i=1,\ldots,n)$, hence all monic monomials, which generate
 $\TGWA{R'}{\si'}{t'}{\mu}$ as a left and as a right $R'$-module, by Lemma \ref{lem:easy}(\ref{it:monomialgeneration}).


\end{proof}

\begin{cor}\label{cor:surjectivity}
If $\varphi:(R,\si,t)\to (R',\si',t')$ is a surjective morphism in $\TGW{n}{\K}$,
then $\Af{\mu}(\varphi):\TGWA{R}{\si}{t}{\mu}\to\TGWA{R'}{\si'}{t'}{\mu}$ is surjective
for any $\mu\in\PM{n}{\K}$.
\end{cor}
\begin{proof}
Follows from Theorem \ref{thm:functoriality}(b) and (c).
\end{proof}

\begin{lem}\label{lem:kernel}
Let $\varphi:(R,\si,t)\to (R',\si',t')$ be a surjective morphism in $\TGW{n}{\K}$.
For $\mu\in\PM{n}{\K}$
put $A=\TGWA{R}{\si}{t}{\mu}$ and let $K\subseteq A_0$ be the kernel of the restriction of $\Af{\mu}(\varphi)$ to $A_0$. Then $\ker\big(\Af{\mu}(\varphi)\big)$ equals the sum of all graded ideals $J$ in $A$ such that $J\cap A_0\subseteq K$.
\end{lem}
\begin{proof}
Put $\bar A=\TGWA{R'}{\si'}{t'}{\mu}$.
Let $I(K)$ denote the sum of all graded ideals in $A$ whose intersection with $A_0$ is
contained in $K$. For brevity we put $\Phi=\Af{\mu}(\varphi)$. Since $\Phi$ is surjective by Corollay \ref{cor:surjectivity}, the image $\Phi\big(I(K)\big)$ is a graded ideal of $\bar A$ whose
intersection with $\bar A_0$ is zero. Lemma \ref{lem:easy}(\ref{it:zerointersection}) applied
to $\bar A$ gives $\Phi\big(I(K)\big)=\{0\}$. Thus $I(K)\subseteq\ker\Phi$.
Conversely, suppose $a\in\ker(\Phi)$. We can assume $a$ is homogeneous. Clearly $(AaA)\cap A_0\subseteq K$ which means that $AaA\subseteq I(K)$. So $\ker(\Phi)\subseteq I(K)$.
\end{proof}

\begin{cor}\label{cor:injectivity}
Let $\mu\in\PM{n}{\K}$.
If $\varphi:(R,\si,t)\to (R',\si',t')$ is a morphism in $\TGW{n}{\K}$
which is injective as a $\K$-algebra morphism, and
 if $(R',\si',t')$ is $\mu$-consistent,
 then $(R,\si,t)$ is $\mu$-consistent and $\Af{\mu}(\varphi)$ is injective.
\end{cor}
\begin{proof}
That $(R,\si,t)$ is $\mu$-consistent is immediate by 
the commutativity of the diagram $\eqref{eq:naturality}$.
Put $A=\TGWA{R}{\si}{t}{\mu}$ and $K=\ker\big(\Af{\mu}(\varphi)|_{A_0}\big)$.
Then by Lemma \ref{lem:easy}(\ref{it:A0}) and commutativity of \eqref{eq:naturality},
$K=\{0\}$.
Thus by Lemma \ref{lem:kernel}, $\Af{\mu}(\varphi)$ is injective.
\end{proof}

\section{Quotients of TGW algebras} \label{sec:tgwaquotients}

The following result generalizes \cite[Proposition 2.12]{BO} from generalized Weyl algebras
to TGW algebras.
If $(R,\si,t)\in\TGW{n}{\K}$, the group $\Z^n$ acts on $R$ via $\si$. An ideal $J$ in $R$
is called \emph{$\Z^n$-invariant} if $\si_g(r)\in J\,\forall r\in J, g\in\Z^n$.
\begin{thm}\label{thm:tgwaquotients}
Let $A=\TGWA{R}{\si}{t}{\mu}$ be a twisted generalized Weyl algebra over $\K$
of degree $n$, and $J$ a $\Z^n$-invariant ideal of $R$.
Let $\bar A=\TGWA{R/J}{\bar\si}{\bar t}{\mu}$, where
$\bar\si_g(r+J)=\si_g(r)+J$ for all $g\in\Z^n, r\in R$
and $\bar t_i=t_i+J$ for all $i$.
Let $\rho:R\to A$ and $\bar\rho:R/J\to \bar A$ be the natural maps.
Put $K=\rho\big(\ker(\bar\rho\circ\pi_J)\big)$, where
 $\pi_J:R\to R/J$ is the canonical projection.
Suppose that
\begin{enumerate}[{\rm (i)}]
\item $\forall g\in\Z^n$ there exists a monic monomial $u_g\in A_g$ such that $A_g=A_0u_g$,
\item $\rho(\si_g(t_i))+K$ is a regular element in $A_0/K$ for all
 $i\in\{1,\ldots,n\},\, g\in\Z^n$. 
\end{enumerate}
Let $\langle K\rangle=AKA$ be the ideal in $\TGWA{R}{\si}{t}{\mu}$ generated by $K$.
Then $AK=\langle K\rangle =KA$, and there is a graded isomorphism
\begin{equation}
\TGWA{R}{\si}{t}{\mu}/\langle K\rangle \simeq \TGWA{R/J}{\bar\si}{\bar t}{\mu}.
\end{equation}
\end{thm}
\begin{proof}
The map $\pi_J$ intertwines the $\Z^n$-actions and maps $t_i$ to $\bar t_i$,
hence it is a
morphism $\pi_J:(R,\si,t)\to (R/J,\bar\si,\bar t)$ in $\TGW{n}{\K}$.
Applying the functor $\Af{\mu}$ from Theorem \ref{thm:functoriality} we get a $\K$-algebra morphism
$\Af{\mu}(\pi_J):\TGWA{R}{\si}{t}{\mu}\to\TGWA{R/J}{\bar\si}{\bar t}{\mu}$
which is surjective by Corollary \ref{cor:surjectivity}.
Thus it remains to prove that the kernel of $\Af{\mu}(\pi_J)$ equals $\langle K\rangle$.

Let $r\in \ker(\bar\rho\circ\pi_J)$.
We claim that $\rho(\si_g(r))\in K$ for all $g\in \Z$.
Let $u_g\in\bar A$ be any monic monomial of degree $g$
and $u_{-g}\in\bar A$ a monic monomial of degree $-g$.
Then $u_gu_{-g}$ can be written as $\bar\rho(a+J)$, where $a$ is  a product of elements of the form $\si_h(t_i)$ ($h\in\Z^n$, $i=1,\ldots,n$), possibly multiplied by an element of $\K^\times$. By assumption (ii), $\rho(a)+K$ is regular in $A_0/K$.
In $\bar A$, we have
$0=u_g\bar\rho(r+J)u_{-g}=\bar\rho(\si_g(r)+J)\bar\rho(a+J)=
\bar\rho(\si_g(r)a+J)$. So $\rho(\si_g(r))\rho(a)\in K$.
Since $\rho(a)+K$ is regular in $A_0/K$ by assumption (ii),
 we get $\rho(\si_g(r))\in K$.

%
%

Using this fact and Lemma \ref{lem:easy}(\ref{it:monomialgeneration}),
we get $AK=\langle K\rangle=KA$.
Since $\langle K\rangle\cap A_0=AK\cap A_0=\big(\sum_{g\in\Z^n} A_gK\big)\cap A_0=A_0K=K$,
Lemma \ref{lem:kernel} implies that
 $\langle K\rangle\subseteq\ker\big(\Af{\mu}(\pi_J)\big)$.
For the converse, let $g\in\Z^n$ and let $a\in A_g\cap\ker\big(\Af{\mu}(\pi_J)\big)$.
By assumption (i), $a=rZ_{i_1}\cdots Z_{i_m}$ for some
 $Z_{i_j}\in\{X_{i_j},Y_{i_j}\}$ and some $r\in A_0$. Put $X_i^\ast=Y_i$, $Y_i^\ast=X_i$.
Then $b:=aZ_{i_m}^\ast\cdots Z_{i_1}^\ast$ has degree zero, hence
 $b\in \ker\big(\Af{\mu}(\pi_j)|_{A_0}\big)$, which by Lemma \ref{lem:easy}(\ref{it:A0}) and
naturality \eqref{eq:naturality} equals $K$.
On the other hand, using relations \eqref{eq:tgwarels}
we get $b=r \cdot \xi u$, where $\xi\in\K^\times$ and $u$ is a product of elements
of the form $\rho(\si_h(t_i))$ ($h\in\Z^n$, $i=1,\ldots,n$).
Since by assumption (ii)
all $\rho(\si_h(t_i))+K$ are regular in $A_0/K$, we get $r\in K$ and thus $a\in KA$.
\end{proof}

\begin{rem}\label{rem:JKidentification}
If $(R,\si,t)$ and $(R/J,\bar \si,\bar t)$ are $\mu$-consistent,
we can identify $R$ and $R/J$ with their images in the corresponding TGW algebras.
Under such identifications, the ideal $K$ in Theorem \ref{thm:tgwaquotients} is
just equal to $J$. This follows from the commutativity of \eqref{eq:naturality}.

\end{rem}

\section{Localizations of TGW algebras}\label{sec:tgwalocalization}

The following trick was first observed in \cite{MPT} in the case of $R$ being
a commutative domain.
\begin{lem}\label{lem:regularshift}
Let $A=\TGWA{R}{\si}{t}{\mu}$ be a twisted generalized Weyl algebra
and $\rho:R\to A$ the natural map.
Suppose all $\rho(t_i)$ are regular in $A_0$. Let $g\in\Z^n$. Then 
\begin{equation}
ab=\si_g(ba)
\end{equation}
for any monic monomial $a\in A_g$ and any $b\in A_{-g}$.
\end{lem}
\begin{proof}
Observe that if $Y_ic\in A_0$ then $Y_ic Y_iX_i=\si_i^{-1}(cY_i)Y_iX_i$, so
since $Y_iX_i=\rho(t_i)$ is regular in $A_0$ we get $Y_ic=\si_i^{-1}(cY_i)$.
Similarly $X_ic=\si_i(cX_i)$ if $X_ic\in A_0$. Applying this for each of the factors in
$a$ the claim follows.
\end{proof}

\begin{thm}\label{thm:tgwalocalization}
Let $A=\TGWA{R}{\si}{t}{\mu}$ be a twisted generalized Weyl algebra over $\K$
of degree $n$ such that $t_1,\ldots,t_n$ are regular in $R$.
Suppose $S\subseteq Z(R)$ is a subset such that
\begin{enumerate}[{\rm (i)}]
\item $S$ is multiplicative: $0\notin S$, $1\in S$ and $r,s\in S\Rightarrow rs\in S$,
\item all elements of $S$ are regular in $R$,
\item $\si_g(S)\subseteq S$ for all $g\in \Z^n$,
\item $(S^{-1}R,\tilde\si,\tilde t)$ is $\mu$-consistent,
\end{enumerate}
where
$\tilde\si_g\in\Aut_\K(S^{-1}R)$ is the unique extension of $\si_g$ for all $g\in \Z^n$,
 and $\tilde t_i=t_i/1$ for all $i=1,\ldots,n$.
By Corollay \ref{cor:injectivity}, $(R,\si,t)$ is also $\mu$-consistent
and we can identify $R$ with its image, $A_0$, under the natural map
 $\rho:R\to\TGWA{R}{\si}{t}{\mu}$.
Then the following statements hold.
\begin{enumerate}[{\rm (a)}]
\item The elements of $S$ are regular in $\TGWA{R}{\si}{t}{\mu}$,
\item $S$ is a left and right Ore set in $\TGWA{R}{\si}{t}{\mu}$,
\item we have an isomorphism
\begin{equation}
S^{-1}\TGWA{R}{\si}{t}{\mu} \simeq \TGWA{S^{-1}R}{\tilde\si}{\tilde t}{\mu}.
\end{equation}
\end{enumerate}
\end{thm}
\begin{proof}
(a) Put $A=\TGWA{R}{\si}{t}{\mu}$.
Suppose $sa=0$ for some $s\in S$, $a\in A\backslash\{0\}$. Without loss of generality we can assume $a$ is homogeneous, say $\deg a=g\in\Z^n$.
By Lemma \ref{lem:easy}(\ref{it:zerointersection})
 and Lemma \ref{lem:easy}(\ref{it:monomialgeneration})
there exist monic monimials $b,c\in A$ such that $bac \in A_0 \setminus\{0\}=R\setminus\{0\}$.
By Lemma \ref{lem:regularshift}, $bac=\si_h(acb)$ where $h=\deg(b)$.
Thus $acb\in R\setminus\{0\}$ also. But we have $sacb=0$, which contradicts that $s$ is regular in $R$.

(b) Easy to check using relations \eqref{eq:tgwarels}.

(c) The canonical map $\varphi:R\to S^{-1}R$ intertwines the $\Z^n$-actions and maps
 $t_i$ to $t_i/1$, hence it is a morphism, $\varphi:(R,\si,t)\to (S^{-1}R,\tilde\sigma,\tilde t)$, in the category $\TGW{n}{\K}$.
Applying the functor $\Af{\mu}$ from Theorem \ref{thm:functoriality}
gives a morphism of graded $\K$-algebras
\begin{equation}
\Af{\mu}(\varphi): \TGWA{R}{\si}{t}{\mu}\to \TGWA{S^{-1}R}{\si}{t}{\mu}.
\end{equation}
Since $\varphi$ is injective, $\Af{\mu}(\varphi)$ is injective
by Corollary \ref{cor:injectivity}.
Since we identify $R$ and $S^{-1}R$ with their images in the respective TGW algebras,
commutativity of 
\eqref{eq:naturality} just says that the restriction of $\Af{\mu}(\varphi)$ to $R$
coincides with $\varphi$.
In particular $\Af{\mu}(\varphi)$ maps each element of $S$ to an invertible element.
Hence, by the universal property of localization, there is an induced map
\begin{equation}
\psi: S^{-1}\TGWA{R}{\si}{t}{\mu}\to \TGWA{S^{-1}R}{\si}{t}{\mu}.
\end{equation}
Since the image of $\psi$ contains $S^{-1}R$ as well as all the generators $X_i, Y_i$,
 $\psi$ is surjective.
Suppose $a\in S^{-1}\TGWA{R}{\si}{t}{\mu}$, $\psi(a)=0$.
By part (a), $S$ consists of regular elements in $\TGWA{R}{\si}{t}{\mu}$ and thus the canonical map $\TGWA{R}{\si}{t}{\mu}\to S^{-1}\TGWA{R}{\si}{t}{\mu}$ is injective and can be regarded as an inclusion. Then $sa\in \TGWA{R}{\si}{t}{\mu}$ for some $s\in S$.
The restriction of $\psi$ to $\TGWA{R}{\si}{t}{\mu}$ coincides with $\Af{\mu}(\varphi)$. So $\Af{\mu}(\varphi)(sa)=\psi(sa)=\psi(s)\psi(a)=0$. Since $\Af{\mu}(\varphi)$ is injective we get $sa=0$, hence $a=0$. This proves that $\psi$ is injective, hence an isomorphism.
\end{proof}

\section{$\mu$-Consistency} \label{sec:muconsistency}
For any $D=(R,\si,t)\in\TGW{n}{\K}$ and any parameter matrix
 $\mu$ over $\K^\times$ of size $n$,
we let $\Lc{\mu,D}$ denote the intersection of all $\Z^n$-invariant
ideals in $R$ containing the set
\begin{gather}\label{eq:NmuDdef}\begin{aligned}
 \Nc{\mu,D} := &\big\{\si_i\si_j(t_it_j)-\mu_{ij}\mu_{ji}\si_i(t_i)\si_j(t_j)\mid i,j=1,\ldots,n, i\neq j\big\}\\
 \cup &\big\{t_j\si_i\si_k(t_j)-\si_i(t_j)\si_k(t_j)\mid i,j,k=1,\ldots,n, i\neq j\neq k\neq i \big\}.
\end{aligned}\end{gather}

Let $T$ be the intersection of all $\Z^n$-invariant multiplicative subsets of $R$
containing $\{t_1,\ldots,t_n\}$.
Concretely, $T$ is the set of all products of elements from the set 
$\big\{\si_g(t_i)\mid g\in\Z^n,\, i\in\{1,\ldots,n\}\big\}$.
For an ideal $I$ of $R$ we let $I^\mathrm{e}$ be the extension of $I$ in $T^{-1}R$,
that is, the ideal in $T^{-1}R$ generated by $I$.
For an ideal $J$ of $T^{-1}R$ we let $J^\mathrm{c}=J\cap R$, called the contraction
of $J$. The extension and contraction operations are order-preserving
and we have $J^\mathrm{ce}=J$ for any ideal $J\subseteq T^{-1}R$ (see for example \cite{AtiMac1994}).

\begin{thm}\label{thm:consistency}
Let $\mu\in\PM{n}{\K}$ and let $D=(R,\si,t)\in\TGW{n}{\K}$ be a regular
TGW datum. Put $\bar R=R/\big((\Lc{\mu,D})^\mathrm{ec}\big)$. Then 
\begin{enumerate}[{\rm (a)}]
\item \label{it:consistency_Lker}
 $(\Lc{\mu,D})^{\mathrm{ec}} = (\ker \rho_{\mu,D})^{\mathrm{ec}}$,
\item \label{it:consistency_bar}
 $(\bar R,\bar\si,\bar t)$ is $\mu$-consistent,
\item \label{it:consistency_JJ}
 $\TGWA{R}{\si}{t}{\mu}/\JJ{\mu,D}\simeq \TGWA{\bar R}{\bar\si}{\bar t}{\mu}$,
\end{enumerate}
where $\JJ{\mu,D}$ is the sum of all graded ideals $J$ in $A=\TGWA{R}{\si}{t}{\mu}$
such that $J\cap A_0\subseteq \rho\big(\ker(\rho_{\mu,D})^{\mathrm{ec}}\big)$.
\end{thm}
\begin{proof} We proceed in steps.
\begin{enumerate}[{\rm (1)}]
\item \label{it:kerinv}
$(\ker \rho_{\mu,D})^{\mathrm{ec}}$ is a $\Z^n$-invariant ideal of $R$. Indeed,
suppose $r\in\ker(\rho_{\mu,D})^\text{ec}$. Then $sr\in\ker\rho_{\mu,D}$ for
some $s\in T$. Let $i\in\{1,\ldots,n\}$.
Then, in $\TGWA{R}{\si}{t}{\mu}$, we
have $0=X_i\rho(sr)Y_i=\rho(\si_i(sr))X_iY_i=\rho (\si_i(r)\si_i(s)\si_i(t_i))$ since $s$ belongs to the center of $R$,
which means that $\si_i(r)\si_i(st_i)\in\ker\rho_{\mu,D}$.
Hence, since $\si_i(st_i)\in T$, we have $\si_i(r)\in (\ker\rho_{\mu,D})^\mathrm{ec}$.
Similarly $\si_i^{-1}(r)\in (\ker\rho_{\mu,D})^\mathrm{ec}$.

\item \label{it:Lsubsetker}
$(\Lc{\mu, D})^\mathrm{ec}\subseteq (\ker \rho_{\mu, D})^\mathrm{ec}$.
By the properties of extension and contraction of ideals,
the claim is equivalent to showing 
$\Lc{\mu, D} \subseteq (\ker \rho_{\mu, D})^\mathrm{ec}$.
By step (1), $(\ker \rho_{\mu, D})^\mathrm{ec}$ is $\Z^n$-invariant.
Thus it is enough to show that
 $\Nc{\mu,D}\subseteq (\ker\rho_{\mu, D})^\mathrm{ec}$.
 For any $i\neq j$ we have in $\TGWA{R}{\si}{t}{\mu}$
\begin{equation}\label{eq:nonlocalizedcommutation}
\si_i\si_j(t_j)X_jX_i=X_jX_i t_j=X_jX_iY_jX_j=\mu_{ij}X_jY_jX_iX_j=
\mu_{ij} \si_j(t_j) X_iX_j
\end{equation}
so that
\begin{equation}\label{eq:LKpftmp}
\si_i\si_j(t_j)\si_i\si_j(t_i) X_jX_i = \mu_{ij}\mu_{ji}\si_j(t_j)\si_i(t_i) X_jX_i.
\end{equation}
Multiplying \eqref{eq:LKpftmp} from the right by $Y_iY_j$ and using that $X_jX_iY_iY_j=\si_j\si_i(t_i)\si_j(t_j)\in T$ we conclude that
\begin{equation}\label{eq:LKproof1}
\si_i\si_j(t_it_j) - \mu_{ij}\mu_{ji}\si_j(t_j)\si_i(t_i)
\in(\ker\rho_{\mu,D})^\mathrm{ec}.
\end{equation}
Let $i,j,k\in\{1,\ldots,n\}$ be three different indices.
%
Using that, as before, $r\in R, rX_iX_jX_k=0\Rightarrow r\in (\ker\rho_{\mu,D})^\mathrm{ec}$, relation \eqref{eq:step2del3ex} implies 
\begin{multline}
\si_j\si_k(t_k) \si_i\si_j\si_k(t_k) \si_i\si_j(t_j) 
\mu_{ij}\si_k\si_j(t_j) \mu_{ik}\si_k(t_k) \mu_{jk}\si_i\si_k(t_k)
\\
-\mu_{jk}\si_k(t_k) \mu_{ik}\si_j\si_k(t_k) \mu_{ij}\si_j(t_j) 
\si_i\si_j\si_k(t_j) \si_i\si_k(t_k) \si_i\si_j\si_k(t_k) 
\in(\ker\rho_{\mu,D})^\mathrm{ec}.
\end{multline}
Dividing by $\mu_{ij}\mu_{jk}\mu_{ik}$ and factorizing we get
\begin{equation}
\big( \si_i\si_j(t_j) \si_k\si_j(t_j) - \si_j(t_j) \si_i\si_j\si_k(t_j) \big) \cdot \si_k(t_k)\si_i\si_k(t_k)\si_j\si_k(t_k)\si_i\si_j\si_k(t_k)
\in(\ker\rho_{\mu,D})^\mathrm{ec}.
\end{equation}
Since $\si_k(t_k)\si_i\si_k(t_k)\si_j\si_k(t_k)\si_i\si_j\si_k(t_k)\in T$ this implies that
\begin{equation}\label{eq:LKproof2}
 \si_i\si_j(t_j) \si_k\si_j(t_j) - \si_j(t_j) \si_i\si_j\si_k(t_j) \in(\ker\rho_{\mu,D})^\mathrm{ecec}=(\ker\rho_{\mu,D})^\mathrm{ec}.
\end{equation}
Applying $\si_j^{-1}$ to \eqref{eq:LKproof2},
 using that $(\ker\rho_{\mu,D})^\mathrm{ec}$ is $\Z^n$-invariant by step (1),
we get together with \eqref{eq:LKproof1} that
 $\Nc{\mu,D}\subseteq (\ker\rho_{\mu,D})^\mathrm{ec}$.
\end{enumerate}

We will now show that we have the following commutative cube:
\begin{gather}\label{eq:cube}
\begin{aligned}
\xymatrix@C=0.3cm@M=1ex{ 
& \bar R \ar@{^{(}->} '[d] [dd]^{\bar\rho} \ar@{_{(}->}[rr]^{\la_{\bar T}}
            && \bar T^{-1}\bar R =\overline{T^{-1}R}\ar@{^{(}->}[dd]^{\hat\rho}    \\
R \ar[dd]^{\rho=\rho_{\mu,D}} \ar@{_{(}->}[rr]^{\qquad\qquad \la_T}
            \ar@{->>}[ur]^{\pi_{(\Lc{\mu,D})^\mathrm{ec}}}
            && T^{-1}R \ar[dd]^(.65){\tilde\rho} \ar@{->>}[ur]^{\pi_{\Lc{\mu,\tilde D}}} &\\
& \TGWA{\bar R}{\bar\si}{\bar t}{\mu} \ar@{^{(}->} '[r] [rr] 
            && \TGWA{\bar T^{-1}\bar R}{\hat\si}{\hat t}{\mu}                    \\
\TGWA{R}{\si}{t}{\mu} \ar[rr] \ar@{->>}[ur]
            && \TGWA{T^{-1}R}{\tilde\si}{\tilde t}{\mu} \ar[ur]^{\simeq}  &
}
\end{aligned}
\end{gather}
Here $\tilde D=(T^{-1}R,\tilde\si,\tilde t)$, $\tilde\si_g\in\Aut_\K(T^{-1}R)$
 is the unique extension
of $\si_g$, $\tilde t_i=t_i/1\in T^{-1}R$, $\bar T$ is the image in $\bar R$ of $T$,
and $\hat\sigma_g$ is the extension of $\bar\sigma_g$ from $\bar R$ to $\bar T^{-1}\bar R$,
$\hat t_i=\bar t_i/1\in\bar T^{-1}\bar R$, and $\bar t_i$ is the image of $t_i$ in $\bar R$.
Note that localizing at $T$ in $T^{-1}R$ has no effect, thus
 $(L_{\mu,\tilde D})^\mathrm{ec}=L_{\mu,\tilde D}$.
The vertical arrows are all instances of the natural transformation $\rho_\mu$.
The maps in the bottom square are the result of applying the functor $\Af{\mu}$
to the respective maps above.
The top square is commutative by the exactness of the localization functor.
Thus the bottom square is commutative by functoriality of $\Af{\mu}$.
Each vertical square is commutative due to the naturality \eqref{eq:naturality}.
Surjectivity of $\Af{\mu}(\pi_{(\Lc{\mu,D})^\mathrm{ec}})$ and
$\Af{\mu}(\pi_{\Lc{\mu,\tilde D}})$ follows by Corollary \ref{cor:surjectivity}.
The localization map $\la_T$ ($\la_{\bar T}$) is injective since elements of
$T$ ($\bar T$) are regular in $R$ ($\bar R$).
It remains to prove injectivity of $\hat\rho, \bar\rho$, $\Af{\mu}(\la_{\bar T})$ and
$\Af{\mu}\big(\pi_{\Lc{\mu,\tilde D}}\big)$.

\begin{enumerate}[{\rm (1)}] \setcounter{enumi}{2}
\item \label{it:hatrhoinj}
 $\hat\rho$ is injective, i.e.
$(\bar T^{-1}\bar R, \hat\si, \hat t)$ is $\mu$-consistent,
To prove this we will use the diamond lemma in
the form of \cite[Theorem 6.1]{B78} and will freely use terminology from that paper.
Put $k=\bar T^{-1}\bar R$.
Let $X=\{X_1^{\pm 1}, \ldots, X_n^{\pm 1}\}$.
Let $M_{X_i^{\pm 1}}$ be the $k$-bimodules which are free of rank $1$
as left $k$-modules: $M_{X_i^{\pm 1}}=k X_i^{\pm 1}$ and with
right $k$-module structure given by
 $X_i^{\pm 1}r=\hat\si_i^{\pm 1}(r)X_i^{\pm 1}$.
Let $k\langle X\rangle$ be the tensor ring on the $k$-bimodule
$M=\bigoplus_{x\in X} M_x$.
Let $I_1$ be the twosided ideal in $k\langle X\rangle$ generated by the set
\begin{multline}\label{eq:diamondgenerators}
\big\{X_i^{\pm 1}X_i^{\mp 1}-1 \mid i=1,\ldots,n\big\} \\ \cup \big\{
X_{i_2}^{\ep_2}X_{i_1}^{\ep_1} -
\hat\si_{i_2}^{(\ep_2-1)/2}\hat\si_{i_1}^{(\ep_1-1)/2}(\tau_{i_2i_1})^{\ep_1\ep_2}
 X_{i_1}^{\ep_1}X_{i_2}^{\ep_2},
\mid \ep_1,\ep_2\in\{1,-1\},\, 1\le i_1<i_2\le n\big \},
\end{multline}
where
\begin{equation}
\tau_{ji}=\frac{\mu_{ij}\hat\si_j(\hat t_j)}{\hat\si_i\hat\si_j(\hat t_j)}\in\bar T^{-1}\bar R.
\end{equation}
One verifies that there are well-defined homomorphisms of $\Z^n$-graded algebras
\begin{gather}
\begin{aligned}
\TGWC{\bar T^{-1}\bar R}{\hat\si}{\hat t}{\mu} &\to  k\langle X\rangle/I_1 \\
X_i &\mapsto X_i\\
Y_i&\mapsto \hat t_iX_i^{-1}\\
r &\mapsto r
\end{aligned}\qquad
\begin{aligned}
 k\langle X\rangle/I_1 &\to \TGWC{\bar T^{-1}\bar R}{\hat\si}{\hat t}{\mu} \\
X_i & \mapsto X_i \\
X_i^{-1}& \mapsto \hat t_i^{-1}Y_i \\
r & \mapsto r
\end{aligned}
\end{gather}
which obviously are inverse to each other.
Moreover they are also morphisms of $k$-bimodules.
So, by Lemma \ref{lem:tgwatgwc}, $\TGWA{\bar T^{-1}\bar R}{\hat\si}{\hat t}{\mu}\simeq k\langle X\rangle /I_1$ as $\Z^n$-graded $\K$-algebras and as $k$-bimodules.
Consider the following reduction system
\begin{align*}
X_{i_2}^{\ep_2}X_{i_1}^{\ep_1} &\longrightarrow 
\si_{i_2}^{(\ep_2-1)/2}\si_{i_1}^{(\ep_1-1)/2}(\tau_{ji}^{\ep_1\ep_2})
  X_{i_1}^{\ep_1}X_{i_2}^{\ep_2},
\qquad \ep_1,\ep_2\in\{1,-1\},\, 1\le i_1<i_2\le n,\\
X_i^{\ep}X_i^{-\ep} &\longrightarrow 1,\qquad \ep\in\{1,-1\},\, i\in\{1,\ldots,n\}.
\end{align*}
These reductions extend uniquely to $k$-bimodule homomorphisms from the appropriate submodules of $k\langle X\rangle$.
Clearly $\langle X\rangle_\mathrm{irr}=\{X_1^{g_1}\cdots X_n^{g_n}\mid g\in\Z^n\}$.
There are no inclusion ambiguities, but four types of overlap ambiguities:
\[X_{i_3}^{\ep_3}X_{i_2}^{\ep_2}X_{i_1}^{\ep_1},\quad
X_{i_2}^{\ep_2}X_{i_1}^{\ep_1}X_{i_1}^{-\ep_1},\quad
X_{i_2}^{\ep_2}X_{i_2}^{-\ep_2}X_{i_1}^{\ep_1},\quad
X_{i}^{\ep}X_{i}^{-\ep}X_{i}^{\ep}\]
the last of which is trivially resolvable. The two middle ones are easily checked
to be resolvable, while the first one is resolvable due to
the identity $\hat\si_k(\tau_{ij})=\tau_{ij}$ if $i,j,k\in\{1,\ldots,n\}$ are different, which hold in $\bar T^{-1}\bar R$ 
since $t_j\si_i\si_k(t_j)-\si_i(t_j)\si_k(t_j)\in (\Lc{\mu,D})^\mathrm{ec}$.
Thus, by \cite[Theorem 6.1]{B78}, the natural map
\[ \bigoplus_{g\in\Z^n} k X_1^{g_1}\cdots X_n^{g_n} \rightarrow k\langle X\rangle /I_1 \]
is an isomorphism of $k$-bimodules, in fact of $\Z^n$-graded $k$-bimodules, since
the reductions are homogeneous. Since $\TGWA{\bar T^{-1}R}{\hat\si}{\hat t}{\mu}$
 is isomorphic to $k\langle X\rangle /I_1$ as graded algebras and as $k$-bimodules,
we in particular have for the degree zero component that the
$k$-bimodule map $k\to k\langle X\rangle /I_1\to
\TGWA{\bar T^{-1}\bar R}{\hat\si}{\hat t}{\mu}_0$
sending $1$ to $1$ is an isomorphism,
which is what we wanted to prove.

\item \label{it:barrhoinj}
$\bar\rho$ and $\Af{\mu}(\la_{\bar T})$ are injective.
Follows from injectivity of $\hat\rho$ and $\la_{\bar T}$
and Corollary \ref{cor:injectivity}. This proves statement (b).

\item \label{it:lociso}
 $\Af{\mu}(\pi_{\Lc{\mu,\tilde D}}): \TGWA{T^{-1}R}{\si}{t}{\mu}\to\TGWA{\bar T^{-1}\bar R}{\si}{t}{\mu}$ is an isomorphism.
Observe that for $\tilde D$ the
 contraction and extension are identity operations, since
the $\tilde t_i$ are already invertible.
Consider $\pi_{\Lc{\mu,\tilde D}}:(T^{-1}R,\si,t)\to (\overline{T^{-1}R},\si,t)$.
As already noted, Corollay \ref{cor:surjectivity} implies that
 $\Af{\mu}(\pi_{\Lc{\mu,\tilde D}})$ is surjective.
The kernel of $\Af{\mu}(\pi_{\Lc{\mu,\tilde D}})
 \circ \tilde\rho$ is by the commutativity of \eqref{eq:cube} equal to
 $\ker(\hat\rho\circ \pi_{\Lc{\mu,\tilde D}})
=\Lc{\mu,\tilde D}$ since $(\bar T^{-1}\bar R,\hat\si,\hat t)$ is $\mu$-consistent by step (4).
But by step (2), $\Lc{\mu,\tilde D}$
is contained in $\ker(\tilde\rho)$.
Thus the restriction of $\Af{\mu}(\pi_{\Lc{\mu,\tilde D}})$ to $\tilde\rho(T^{-1}R)$
 is injective,
hence $\Af{\mu}(\pi_{\Lc{\mu,\tilde D}})$ is injective by Lemma \ref{lem:kernel}.

\item \label{it:kersubsetL}
 $\ker(\rho_{\mu,D})^\mathrm{ec}\subseteq \Lc{\mu,D}^\mathrm{ec}$.
By \eqref{eq:cube} and injectivity of $\bar\rho$ we have $\ker(\rho_{\mu,D}) \subseteq 
\ker\big( \Af{\mu}(\pi_{(L_{\mu,D})^\mathrm{ec}}) \circ \rho_{\mu,D}\big)
=\ker(\bar\rho \circ\pi_{(L_{\mu,D})^\mathrm{ec}}\big)
=(L_{\mu,D})^\mathrm{ec}$. This proves statement (a).

\item \label{it:JJ}
 $\ker \big(\Af{\mu}(\pi_{(\Lc{\mu,D})^\mathrm{ec}})\big) = \JJ{\mu,D}$.
Follows from Lemma \ref{lem:kernel} and that
the kernel of $\Af{\mu}(\pi_{(\Lc{\mu,D})^\mathrm{ec}})\circ\rho_{\mu,D}$
is equal to the kernel of $\pi_{(\Lc{\mu,D})^\mathrm{ec}}$ which is
$(\Lc{\mu,D})^\mathrm{ec}=(\ker\rho_{\mu,D})^\mathrm{ec}$, by step (2) and step (6).
This proves statement (c).
\end{enumerate}

\end{proof}

We can now deduce the main result of the paper. 
%
\begin{proof}[Proof of Theorem~A]
Let $D$ be the TGW datum $(R,\si,t)$ and 
$\rho_{\mu,D}:R\to\TGWA{R}{\si}{t}{\mu}$ be the canonical map of $R$-rings.
 By Theorem \ref{thm:consistency}\eqref{it:consistency_Lker}, $\rho_{\mu,D}$ is injective if and only if $L_{\mu,D}=\{0\}$, which by definition of $L_{\mu,D}$ is equivalent to that the set $N_{\mu,D}$, defined in \eqref{eq:NmuDdef}, is equal to $\{0\}$. But $N_{\mu,D}=\{0\}$ just means that relations \eqref{eq:TGWA_consistency_relations1} and \eqref{eq:TGWA_consistency_relations2} hold. This proves the required equivalence.

It only remains to prove the independence of the two conditions \eqref{eq:TGWA_consistency_relations1} and \eqref{eq:TGWA_consistency_relations2}.
It is clear that \eqref{eq:TGWA_consistency_relations1} does not follow from
\eqref{eq:TGWA_consistency_relations2} because the latter does not depend on $\mu_{ij}$.
In Example \ref{ex:trivialtgwa_1}, condition \eqref{eq:TGWA_consistency_relations1}
is satisfied but $\rho_{\mu,D}:R\to\TGWA{R}{\si}{t}{\mu}$ is the zero map, and hence the TGW datum $(R,\si,t)$ is not $\mu$-consistent. Thus, by
Theorem A, condition \eqref{eq:TGWA_consistency_relations2} cannot hold (this can also be verified directly; relation \eqref{eq:trivialtgwa_invertible} shows that $N_{\mu,D}$ contains a nonzero, even invertible, element). This proves that condition \eqref{eq:TGWA_consistency_relations2} does
not follow from \eqref{eq:TGWA_consistency_relations1} either.
\end{proof}


We also obtain the following corollaries.
\begin{cor}\label{cor:RR'consistency}
Assume $\varphi:(R,\si,t)\to (R',\si',t')$ is a morphism between
regular TGW data. Let $\mu\in\PM{n}{\K}$ and suppose
$(R,\si,t)$ is $\mu$-consistent. Then $(R',\si',t')$ is also $\mu$-consistent.
\end{cor}
\begin{proof} Follows from that $\varphi(N_{\mu,(R,\si,t)})=N_{\mu,(R',\si',t')}$.
\end{proof}

\begin{cor}\label{cor:consistencylocalization}
Let $\mu\in\PM{n}{\K}$. If $(R,\si,t)\in\TGW{n}{\K}$ is regular and
$\mu$-consistent,
then $\Af{\mu}(\la_S):\TGWA{R}{\si}{t}{\mu}\to\TGWA{S^{-1}R}{\tilde\si}{\tilde t}{\mu}$
is injective for any regular $\Z^n$-invariant multiplicative set $S\subseteq R$.
\end{cor}
\begin{proof} Use Corollary \ref{cor:RR'consistency} and
Corollary \ref{cor:injectivity}.
\end{proof}

\section{Weak $\mu$-consistency and crossed product algebra embeddings}
\label{sec:weaklymuconsistent}

In this section we show that sometimes a TGW datum $(R,\si,t)$ which
is not $\mu$-consistent may nevertheless be replaced by another TGW datum $(R',\si',t')$ which is $\mu$-consistent and such that the corresponding TGW algebras are isomorphic. Such TGW data will be called weakly $\mu$-consistent (see definition below).

\begin{thm}\label{thm:muequiv}
Let $\mu\in\PM{n}{\K}$ and let $D=(R,\si,t)\in\TGW{n}{\K}$ be regular.
Then the following statements are equivalent.
\begin{enumerate}[{\rm (i)}]
\item \label{it:muequiv_rhotireg}
  $\rho_{\mu,D}(t_i)$ is regular in $\rho_{\mu,D}(R)$ for all $i$,
\item \label{it:muequiv_kerrho}
 $\big(\ker \rho_{\mu,D}\big)^{\mathrm{ec}}=\ker \rho_{\mu,D}$,
\item \label{it:muequiv_bariso}
 $\Af{\mu}(\pi_{L_{\mu,D}^\mathrm{ec}}):\TGWA{R}{\si}{t}{\mu}\to
\TGWA{\bar R}{\bar \si}{\bar t}{\mu}$ is an isomorphism,
\item \label{it:muequiv_lociso}
 $\Af{\mu}(\la_T):\TGWA{R}{\si}{t}{\mu}\to
       \TGWA{T^{-1}R}{\tilde\sigma}{\tilde t}{\mu}$ is injective,
\item \label{it:muequiv_almost}
 there exists a morphism $\psi:(R,\si,t)\to (R',\si',t')$ from $(R,\si,t)$ to a $\mu$-consistent regular TGW datum $(R',\si',t')\in\TGW{n}{\K}$ such that
$\Af{\mu}(\psi):\TGWA{R}{\si}{t}{\mu}\to\TGWA{R'}{\si'}{t'}{\mu}$ is an isomorphism.
\end{enumerate}
\end{thm}
\begin{proof}
That (\ref{it:muequiv_rhotireg}) $\Leftrightarrow$ (\ref{it:muequiv_kerrho}) is immediate
by definition.
The equivalence (\ref{it:muequiv_kerrho}) $\Leftrightarrow$ (\ref{it:muequiv_bariso}) follows from Theorem \ref{thm:consistency}(\ref{it:consistency_JJ}).
The equivalence (\ref{it:muequiv_bariso}) $\Leftrightarrow$ (\ref{it:muequiv_lociso})
follows from the cube \eqref{eq:cube}. Trivially (\ref{it:muequiv_bariso}) implies
(\ref{it:muequiv_almost}) because $(\bar R,\bar \si,\bar t)$ is $\mu$-consistent by
Theorem \ref{thm:consistency}(\ref{it:consistency_bar}).
We prove (\ref{it:muequiv_almost})$\Rightarrow$(\ref{it:muequiv_rhotireg}).
Let $r\in R$, $i\in\{1,\ldots,n\}$
and assume that $\rho_{\mu,D}(rt_i)=0$. Thus $\Af{\mu}(\psi)(\rho_{\mu,D}(rt_i))=0$.
By naturality \eqref{eq:naturality} we get $\rho_{\mu,D'}(\psi(rt_i))=0$,
where $D'=(R',\si',t')$. But
$\psi(t_i)=t_i'$ and $\rho_{\mu,D'}$ is injective so we get $\psi(r)t_i'=0$. Since
$D'$ is regular we have $\psi(r)=0$. Hence, by naturality,
 $\Af{\mu}(\psi)(\rho_{\mu,D}(r))=0$, thus $\rho_{\mu,D}(r)=0$ since $\Af{\mu}(\psi)$
 is bijective. This proves that
$\rho_{\mu,D}(t_i)$ is regular in $\rho_{\mu,D}(R)$.
\end{proof}
Theorem \ref{thm:muequiv}(\ref{it:muequiv_almost}) motivates the following definition.
\begin{dfn}
A regular TGW datum satisfying the equivalent conditions in Theorem \ref{thm:muequiv}
is called \emph{weakly $\mu$-consistent}.
\end{dfn}

As an application, we prove that (weakly) $\mu$-consistent TGW algebras can be embedded into crossed product algebras.

Recall that a group graded algebra $S=\bigoplus_{g\in G} S_g$
such that each $S_g$ ($g\in G$) contains an invertible element
is called a \emph{$G$-crossed product over $S_e$}. One can show
(see for example \cite{NasVan2004}) that any $G$-crossed product $S$
is isomorphic as a left
$S_e$-module to the group algebra $S_e[G]=\bigoplus_{g\in G} S_eu_g$
with product given by
\[s_1u_gs_2u_h = s_1\zeta_g(s_2)\alpha(g,h)u_{gh}\]
for some unique maps
$\zeta:G\to\Aut(S)$ and $\alpha:G\times G\to (S_e)^\times$
satisfying
\begin{subequations}\label{eq:crossedproducteqs}
\begin{gather}
\zeta_g (\zeta_h (a)) = \al(g,h)\zeta_{gh}(a)\al(g,h)^{-1}\\
\al(g,h)\al(gh,k)=\zeta_g(\al(h,k))\al(g,hk)\\
\al(g,e)=\al(e,g)=1
\end{gather}
\end{subequations}
for all $g,h,k\in G$, $a\in S_e$, where $e\in G$ is the identity element.
Then $S$ is denoted $S_e\rtimes_\alpha^\zeta G$.

\begin{thm}\label{thm-main-1}
\begin{enumerate}[{\rm (a)}]
\item
Let $\mu\in\PM{n}{\K}$ and $D=(R,\si,t)\in\TGW{n}{\K}$. Assume $t_1,\ldots,t_n\in R^\times$.
Then, if $D$ is $\mu$-consistent (equivalently, if \eqref{eq:TGWA_consistency_relations1} and
\eqref{eq:TGWA_consistency_relations2} hold), there is a unique $\si$-twisted $2$-cocycle
$\al:\Z^n\times\Z^n\to R^\times$ for which there exists a graded $\K$-algebra isomorphism
\begin{align*}
\xi_{\mu,D}: R\rtimes_\al^\si \Z^n &\to \TGWA{R}{\si}{t}{\mu}\\
\intertext{satisfying}
\xi(ru_g)&= rX_1^{g_1}\cdots X_n^{g_n}.
\end{align*}

\item
If $\mu\in\PM{n}{\K}$ and  $(R,\si,t)\in\TGW{n}{\K}$ is regular and
weakly $\mu$-consistent, then $\TGWA{R}{\si}{t}{\mu}$ can be embedded into
a $\Z^n$-crossed product algebra:
\begin{equation*}
\xymatrix@C=1.2cm@M=1ex{ 
\TGWA{R}{\si}{t}{\mu} \ar[r]^{\Af{\mu}(\pi_{\Lc{\mu,D}})}_{\simeq} &
\TGWA{\bar R}{\bar\si}{\bar t}{\mu}
    \ar@{^{(}->}[r]^{\Af{\mu}(\la_{\bar T})}  &
\TGWA{\bar T^{-1}\bar R}{\hat\si}{\hat t }{\mu} \ar[r]^{(\xi_{\mu,\hat D})^{-1}}_{\simeq} &
\bar T^{-1}\bar R\rtimes_{\al}^{\hat\si} \Z^n,
}
\end{equation*}
where
$\bar R=R/(\Lc{\mu,D})^\mathrm{ec}$,
 $\bar\si_g(r+(\Lc{\mu,D})^\mathrm{ec})=\si_g(r)+(\Lc{\mu,D})^\mathrm{ec}$,
$\bar t_i=t_i+(\Lc{\mu,D})^\mathrm{ec}$,
$\bar T$ is the image of $T$ in $\bar R$,
$\la_{\bar T}:\bar R\to \bar T^{-1}\bar R$ is the localization map,
$\hat D=(\bar T^{-1}\bar R,\hat\si,\hat t)$, 
$\hat\si_g(s^{-1}r)=\bar\sigma_g(s)^{-1}\bar\sigma_g(r)$ for $s\in\bar T$, $r\in \bar R$.
$\hat t_i=\bar t_i/1$ 
and the ideal $(\Lc{\mu,D})^\mathrm{ec}$ and the set $T$ were
 defined in the beginning of Section \ref{sec:muconsistency}.
\end{enumerate}
\end{thm}
\begin{proof}
(a) Follows by Lemma \ref{lem:tgwatgwc} and the above comments on crossed products.

(b) By definition of weakly $\mu$-consistency, $\Af{\mu}(\pi_{\Lc{\mu,D}})$ is
an isomorphism. By Theorem \ref{thm:consistency}(\ref{it:consistency_bar}),
$\TGWA{\bar R}{\bar\si}{\bar t}{\mu}$ is $\mu$-consistent. Corollary
\ref{cor:consistencylocalization} implies that $\Af{\mu}(\la_{\bar T})$ 
is injective. By Corollary \ref{cor:RR'consistency}, $\hat D$ is $\mu$-consistent
so the last isomorphism follows from part (a).
\end{proof}

\end{document}